\documentclass[11pt,a4paper]{article}
\usepackage[T1]{fontenc}
\usepackage{lmodern}
\usepackage[margin=28mm]{geometry}
\usepackage{amsmath,amssymb,amsthm,mathtools}
\usepackage{microtype}
\usepackage[hidelinks]{hyperref}
\hypersetup{pdftitle={The small Davenport constant of H27 times C3 to the r},
pdfauthor={Andreas Volkmann}}
\newtheorem{theorem}{Theorem}[section]
\newtheorem{lemma}[theorem]{Lemma}
\newtheorem{corollary}[theorem]{Corollary}
\theoremstyle{remark}

\newcommand{\F}{\mathbb F_3}
\newcommand{\dc}{\mathsf d}
\newcommand{\UT}{\operatorname{UT}}
\newcommand{\rad}{\operatorname{rad}}

\title{The small Davenport constant of\texorpdfstring{\\ $H_{27}\times C_3^r$}{ H27 times C3 to the r}}
\author{Andreas Volkmann}
\date{17 September 2026}
\begin{document}
\maketitle
\begin{abstract}
Let $H_{27}=\UT_3(\mathbb F_3)$ be the nonabelian group of order $27$
and exponent $3$. We prove that
\[
\dc(H_{27}\times C_3^r)=2r+6\qquad\text{for every integer }r\geq0.
\]
The proof combines an affine coefficient identity in the group algebra of an
elementary abelian group with a decomposition of the nonorthogonality graph
of $\mathbb F_3^2$ into eight edge-disjoint zero-sum triangles. It is uniform
in $r$, does not use the value of the small Davenport constant for a smaller
nonabelian group, and requires no computational enumeration.
\end{abstract}

\section{Introduction}
For a finite group $G$, a sequence over $G$ is a finite unordered list of
elements of $G$, with repetitions allowed. A nonempty subsequence is
\emph{product-one} if its terms can be ordered so that their product is the
identity. A sequence is \emph{product-one-free} if it has no such subsequence.
The small Davenport constant $\dc(G)$ is the maximum length of a
product-one-free sequence over $G$.

For an elementary abelian $p$-group, Olson's theorem gives
$\dc(C_p^n)=n(p-1)$~\cite{Olson}.
For the exponent-$p$ Heisenberg group $H_{p^3}=\UT_3(\mathbb F_p)$,
the value $\dc(H_{27})=6$ was determined computationally by Cziszter,
Domokos and Sz\"{o}ll\H{o}si~\cite{CDS} and subsequently proved
theoretically by Godara and Sarkar~\cite{GS}. A uniform proof of
$\dc(H_{p^3})=3p-3$ for all odd primes $p$ was given in~\cite{Volkmann}.
The present paper considers adjoining arbitrarily many elementary abelian
central direct factors at the prime $3$.

\begin{theorem}\label{thm:main}
For every integer $r\geq0$,
\[
\dc(H_{27}\times C_3^r)=2r+6.
\]
\end{theorem}

The lower bound is immediate from a basis sequence with each term repeated
twice. The upper bound uses the additive vector space underlying suitable
coordinates for the group. For a hypothetical product-one-free sequence of
length $2(r+3)+1$, a high-degree augmentation product gives a coefficient
identity for every noncommuting pair. Each nonempty additive zero-sum
subsequence has a noncommutation graph consisting of one triangle and
isolated vertices. We assign weights to pairs so that each such triangle
has total weight one. Summing the pair identities then contradicts the
coefficient identity for the whole sequence.

All subsets used below refer to indexed occurrences of sequence terms.
Thus equal group elements in different positions are always treated as
distinct occurrences. All coefficient identities are over $\F$.

\section{Coordinates and ordering corrections}
Put $m=r+3$ and write
\[
E=\F^2\oplus\F z\oplus\F^r,
\qquad x=(a,b,c,t),\qquad z=(0,0,1,0).
\]
Define
\[
\pi(x)=(a,b),\qquad
\omega(x,y)=a b'-a'b
\quad\text{for }y=(a',b',c',t').
\]
On $E$, set
\begin{equation}\label{eq:law}
x*y=x+y+\tfrac12\omega(x,y)z.
\end{equation}
The scalar $1/2$ denotes the inverse of $2$ in $\F$.
This group is isomorphic to $H_{27}\times C_3^r$ under
\[
(a,b,c,t)\longmapsto
\left(\begin{pmatrix}
1&a&c+\tfrac12 ab\\
0&1&b\\
0&0&1
\end{pmatrix},t\right).
\]
Indeed, matrix multiplication yields~\eqref{eq:law} after subtracting
$\tfrac12(a+a')(b+b')$ from the top-right entry. Its identity is the zero
vector. Two terms commute if and only if their $\omega$-pairing is zero.
The alternating form has radical $\F z\oplus\F^r$, and its induced form
on $E/\rad(\omega)\cong\F^2$ is the determinant.

Let $S=(x_1,\ldots,x_N)$ be a sequence in these coordinates. For an index
set $A\subseteq[N]$, put $s(A)=\sum_{i\in A}x_i$, where the sum is additive
in $E$. Repeated use of~\eqref{eq:law} shows that for any ordering
$i_1,\ldots,i_k$ of $A$,
\begin{equation}\label{eq:order}
x_{i_1}*\cdots*x_{i_k}
=s(A)+\frac12\sum_{u<v}\omega(x_{i_u},x_{i_v})z.
\end{equation}
Let $\Omega(A)\subseteq\F$ be the set of scalar correction values in this
formula as the ordering varies.

\begin{lemma}\label{lem:local}
Suppose that $S$ is product-one-free. Then:
\begin{enumerate}
\item If $s(A)\in\{z,-z\}$, all terms indexed by $A$ commute pairwise.
\item If $A\ne\varnothing$ and $s(A)=0$, the terms indexed by $A$ do not
all commute pairwise.
\item If $s(A)\in\F z$, the set $A$ cannot contain two disjoint
noncommuting pairs of occurrences.
\end{enumerate}
\end{lemma}
\begin{proof}
Reversing an ordering negates its correction, so
$\Omega(A)=-\Omega(A)$. If $A$ contains a noncommuting pair, place its two
terms next to each other and exchange them. The correction changes by a
nonzero scalar. Hence $\Omega(A)$ has at least two elements and, being
symmetric in $\F$, contains both $1$ and $-1$. This proves (1).
If all pairings vanish, every correction is zero, proving (2).

For (3), place each of two disjoint noncommuting pairs in adjacent positions
in one ordering. Their exchanges can be made independently. The four
corrections contain a translate of
$\{0,u\}+\{0,v\}=\F$ for some $u,v\ne0$.
If $s(A)\in\F z$, one of these corrections produces the identity.
\end{proof}

\section{An affine coefficient identity}
The group algebra used in this section is the \emph{commutative} group
algebra of the additive group of $E$, not the group algebra of~\eqref{eq:law}.
Write its basis elements as $X^x$, so $X^xX^y=X^{x+y}$, and let $I$ be its
augmentation ideal.

\begin{lemma}\label{lem:affine}
Let $R=(y_1,\ldots,y_{2m-1})$ be any sequence in $E\cong\F^m$.
The coefficient function
\[
f_R(w)=[X^w]\prod_{j=1}^{2m-1}(1-X^{y_j})
\]
is affine on $E$. Moreover, $I^{2m+1}=0$.
\end{lemma}
\begin{proof}
Choose a basis $e_1,\ldots,e_m$ of $E$ and set $T_j=X^{e_j}-1$. Then
\[
\F[E]\cong
\F[T_1,\ldots,T_m]/(T_1^3,\ldots,T_m^3),
\qquad I=(T_1,\ldots,T_m).
\]
The residue classes of the monomials with all exponents at most two form a
basis. In particular, $I^{2m+1}=0$, and $I^{2m-1}$ is spanned by
\[
M=\prod_{j=1}^m T_j^2,
\qquad
M_i=T_i\prod_{j\ne i}T_j^2\quad(1\leq i\leq m).
\]
In one coordinate, $T_j^2=1+X^{e_j}+X^{2e_j}$ has constant coefficient
function one. The coefficient function of $T_i$ at coordinate values
$0,1,2$ is $(-1,1,0)$, namely $w_i\mapsto-1-w_i$.
Consequently, the coefficient functions of $M$ and $M_i$ are respectively
$1$ and $-1-w_i$. Each factor $1-X^{y_j}$ belongs to $I$, so the stated
product lies in $I^{2m-1}$ and has an affine coefficient function.
\end{proof}

\begin{lemma}\label{lem:pair}
Suppose that $S=(x_1,\ldots,x_{2m+1})$ is product-one-free.
For every pair $i<j$ with $\omega(x_i,x_j)\ne0$,
\begin{equation}\label{eq:pair}
\sum_{\substack{A\subseteq[2m+1]\\s(A)=0,\ \{i,j\}\subseteq A}}
(-1)^{|A|}=0.
\end{equation}
\end{lemma}
\begin{proof}
Remove the occurrences $i,j$ from $S$ to obtain a sequence $R$ of length
$2m-1$. No subsequence of $R$ has sum $z-x_i-x_j$ or $-z-x_i-x_j$:
adjoining $i,j$ would contradict Lemma~\ref{lem:local}(1).
Thus the coefficient function in Lemma~\ref{lem:affine} vanishes at both
of these points. Since it is affine,
\[
f_R(-x_i-x_j)
=\tfrac12\bigl(f_R(z-x_i-x_j)+f_R(-z-x_i-x_j)\bigr)=0.
\]
Expanding the product defining $f_R$ gives~\eqref{eq:pair}; adding two
occurrences does not change the sign.
\end{proof}

\section{Triangle weights and the upper bound}
Consider the graph on the eight nonzero vectors of $P=\F^2$ in which
$u,v$ are joined when $\det(u,v)\ne0$. Every edge belongs to exactly one
zero-sum triangle:
\begin{equation}\label{eq:triangle}
\{u,v,-u-v\}.
\end{equation}
The three vertices in~\eqref{eq:triangle} are distinct and pairwise
nonorthogonal. Each of its edges recovers the same third vertex.
Every vertex has six neighbors, so the graph has $24$ edges, partitioned
into eight such triangles.

Select one edge from each triangle and define a symmetric weight function
$\rho:P\times P\to\F$ by assigning weight one to the selected edges and
zero to all other pairs. In particular, orthogonal pairs receive weight
zero. This choice can be made without computation: fix any total order
on $P$ and choose the lexicographically first edge of each triangle.
For each triangle~\eqref{eq:triangle}, the sum of its three edge weights
is one.

\begin{lemma}\label{lem:weight}
Suppose that $S$ is product-one-free. If $A\ne\varnothing$ and $s(A)=0$,
then
\begin{equation}\label{eq:weight}
\sum_{\substack{i,j\in A\\i<j}}\rho(\pi(x_i),\pi(x_j))=1.
\end{equation}
\end{lemma}
\begin{proof}
Make a graph on the occurrences in $A$, joining $i$ and $j$ when
$\omega(x_i,x_j)\ne0$. By Lemma~\ref{lem:local}, it has at least one edge
and no two disjoint edges. No vertex can have degree one: if $j$ were the
unique neighbor of $i$, then
\[
0=\omega(x_i,s(A))=\omega(x_i,x_j)\ne0.
\]

A graph with at least one edge and no two disjoint edges consists of a star
or a triangle together with isolated vertices. To see this, if two edges
are $ab,ac$ and all edges contain $a$, the graph is a star. Otherwise an
edge avoiding $a$ must be $bc$; every edge must then lie in $\{a,b,c\}$.
The single-edge case is also a star. The absence of degree-one vertices
therefore leaves exactly one triangle together with isolated vertices.

The projections of two adjacent triangle terms form a basis of $P$.
Every isolated term is orthogonal to both, so its projection is zero.
Since $s(A)=0$, the three triangle projections sum to zero and therefore
form one of~\eqref{eq:triangle}. All other pairs have zero weight, proving
\eqref{eq:weight}.
\end{proof}

\begin{proof}[Proof of Theorem~\ref{thm:main}]
Write $G=H_{27}\times C_3^r$. For the upper bound, suppose that
$S=(x_1,\ldots,x_N)$ is product-one-free
with $N=2m+1$. By Lemma~\ref{lem:affine},
\[
\prod_{i=1}^{N}(1-X^{x_i})=0.
\]
Taking the coefficient of $X^0$ and separating the empty subset yields
\begin{equation}\label{eq:total}
\sum_{\substack{\varnothing\ne A\subseteq[N]\\s(A)=0}}(-1)^{|A|}=-1.
\end{equation}
On the other hand, multiply~\eqref{eq:pair} by
$\rho(\pi(x_i),\pi(x_j))$ and sum over noncommuting pairs. Nonzero weights
occur only for such pairs. Interchanging the finite sums and using
Lemma~\ref{lem:weight}, we obtain
\begin{align*}
0
&=\sum_{\substack{\varnothing\ne A\subseteq[N]\\s(A)=0}}
(-1)^{|A|}
\sum_{\substack{i,j\in A\\i<j}}\rho(\pi(x_i),\pi(x_j))\\
&=\sum_{\substack{\varnothing\ne A\subseteq[N]\\s(A)=0}}(-1)^{|A|},
\end{align*}
contrary to~\eqref{eq:total}. Hence $\dc(G)\leq2m$.

For the lower bound, choose $u=(1,0,0,0)$ and $v=(0,1,0,0)$, and let
$c_1,\ldots,c_r$ be the standard basis of the last direct summand.
Write $x^{[2]}$ for two occurrences of $x$ in a sequence. Then
\[
u^{[2]}\boldsymbol\cdot v^{[2]}\boldsymbol\cdot z^{[2]}
\boldsymbol\cdot c_1^{[2]}\boldsymbol\cdots\boldsymbol\cdot c_r^{[2]}
\]
is product-one-free. Indeed, a product-one subsequence must first have
zero image in $E/\F z$. The images of $u,v,c_1,\ldots,c_r$ are independent,
and their selected multiplicities lie in $\{0,1,2\}$, so none can occur.
The selected multiplicity of $z$ must then also be zero. Only the empty
subsequence could have product one. The displayed sequence has length
$2r+6=2m$, proving the result.
\end{proof}

\section{Structural formulation and scope}
\begin{corollary}\label{cor:structural}
Let $G$ be a finite group of exponent $3$ with $|G:Z(G)|=9$. Then
$G\cong H_{27}\times C_3^r$ for some $r\geq0$, and
\[
\dc(G)=2\log_3|G|.
\]
\end{corollary}
\begin{proof}
The quotient $G/Z(G)$ is elementary abelian of order nine, so $G$ has
nilpotency class at most two. Choose $u,v$ whose images form a basis of
this quotient. Their commutator $z=[u,v]$ is nontrivial, since otherwise
$G$ would be abelian. Because commutators are central and all elements
have order dividing three, $K=\langle u,v\rangle$ has the $27$ distinct
normal forms $u^a v^b z^c$, where $a,b,c\in\{0,1,2\}$. Thus
$K\cong H_{27}$ and $K\cap Z(G)=\langle z\rangle$.
The center is elementary abelian; choose a complement $C$ to
$\langle z\rangle$ in $Z(G)$. Since $G=KZ(G)$, we obtain $G=K\times C$.
Theorem~\ref{thm:main} now applies.
\end{proof}

The rank-two condition on the scalar alternating form is used in
Lemma~\ref{lem:weight}: a vector orthogonal to two independent vectors in
$\F^2$ is zero. In a symplectic space of dimension four, their common
orthogonal complement can be nonzero. Isolated vertices may then
contribute to the additive sum, so the nonorthogonal triangle need not
itself be zero-sum. The plane weighting argument consequently does not
by itself prove the analogous formula in higher rank.

Theorem~\ref{thm:main} makes no assertion for other primes, for higher
commutator rank, or for a commutator subgroup of dimension greater than
one. A natural further target is
\[
\dc(H_2(\mathbb F_3)\times C_3^r)=2r+10\qquad(r\geq0),
\]
where $H_2(\mathbb F_3)$ is the extraspecial group of order $3^5$ and
exponent three. This further statement is not proved here.

\paragraph{Use of AI-assisted tools.}
OpenAI's ChatGPT was used to develop proof strategies, formulate and
check intermediate arguments, search the literature, and prepare the
manuscript. These AI-assisted checks do not constitute independent human
peer review. The proof presented here is self-contained and does not rely
on computational enumeration.

\end{document}